\documentclass[11pt,reqno]{amsart}
\usepackage{amsmath,amstext, amsthm,amsfonts,amssymb,amsthm} 
\numberwithin{equation}{section}
\def\D{\mathcal{D}}

\def\ed{\mathrm{ed}}
\def\cro{\mathrm{cr}}
\def\c{\mathrm{c}}

\usepackage{amsthm,amsmath,amssymb,latexsym,url}
\newtheorem{thm}{Theorem}
\newtheorem{prop}[thm]{Proposition}

\newtheorem{cor}[thm]{Corollary}
\newtheorem{defn}[thm]{Definition}
\newtheorem{lem}[thm]{Lemma}

\theoremstyle{remark}

\def\F{{\mathcal F}}
\def\la{\lambda}

\begin{document}
\title{A  curious $q$-analogue of Hermite polynomials}
\thanks{\today}
\author{Johann Cigler}
\address{Institut  f\"ur Mathematik, Universit\"at Wien \\
A-1090 Wien, Osterreich}
\email{johann.cigler@univie.ac.at}

\author{Jiang Zeng}
\address{Universit\'{e} de Lyon \\ Universit\'{e} Lyon 1 \\ Institut Camille Jordan \\ UMR 5208 du CNRS \\
43, boulevard du 11 novembre 1918 \\ F-69622 Villeurbanne Cedex, France}
\email{zeng@math.univ-lyon1.fr}



\begin{abstract} Two well-known $q$-Hermite polynomials are the continuous and discrete $q$-Hermite polynomials.
In this paper we consider a new family of $q$-Hermite polynomials and prove several curious properties about these polynomials.
One striking property is the connection with $q$-Fibonacci and $q$-Lucas polynomials. The latter relation yields a generalization of 
the Touchard-Riordan formula.
\end{abstract}

\maketitle
\tableofcontents
\section{Introduction}
There are many ways to construct Hermite polynomials. 
For example, the normalized Hermite polynomials $H_n(x,s)=s^{n/2}H_n(x/\sqrt{s}, 1)$ ($n\geq 0$) may be defined by
the recurrence relation:
\begin{align}\label{eq:rec1}
H_{n+1}(x,s)=xH_n(x,s)-nsH_{n-1}(x,s),
\end{align}
with initial values $H_0(x,s)=1$ and $H_{-1}(x,s)=0$.
By induction, we have 
\begin{align}\label{eq:op}
H_n(x,s)=(x-s{\D})^n\cdot 1,
\end{align}
where ${\D}=\frac{d}{dx}$ denotes the differentiation operator. It follows that 
\begin{align}\label{eq:rec2}
{\D}\,H_n(x,s)=nH_{n-1}(x,s).
\end{align}
They have  the explicit formula (see \cite[Chapter 6]{AAR})  
$$
H_n(x,s)=\sum_{k=0}^n{n\choose 2k}(-s)^k(2k-1)!! x^{n-2k}.
$$
The first terms are
$$
1,\;x, \; -s+x^2,\; -3sx+x^3,\; 3s^2-6sx^2+x^4,\; 15s^2x-10sx^3+x^5.
$$
The Hermite polynomials are orthogonal with respect  to the the linear functional defined by the moments
$$
\mu_n=\frac{1}{\sqrt{2\pi}}\int_{-\infty}^{\infty} x^ne^{-x^2/2}dx=
\left\lbrace\begin{array}{cc}
(n-1)!!& \text{if $n$ is even},\\
0& \text{otherwise}.
\end{array}\right.
$$
In other words,  the $n$-th moment $\mu_n$ of the measure of  the Hermite polynomials is the number of the {\em complete matchings} 
on $[n]:=\{1,\ldots, n\}$, i.e., $\mu_{2n}=(2n-1)!!$ and $\mu_{2n+1}=0$.

Consider the rescaled Hermite polynomials $p_n(z,x,s)=H_n(z-x, -s)$ defined by
\begin{align}
p_{n+1}(z,x,s)=(z-x)p_n(z,x,s)+snp_{n-1}(z,x,s)
\end{align}
with initial values $p_0(z,x,s)=1$ and $p_{-1}(z,x,s)=0$.
Let $\mathcal{F}$ be the linear functional on the polynomials in $z$ defined by
$\mathcal{F}(p_n(z,x,s))=\delta_{n,0}$. Then, the moments
 $\mathcal{F}(z^n)$ are again the Hermite polynomials
\begin{align}
\mathcal{F}(z^n)=(\sqrt{-s})^n 
\sum_{k=0}^n {n\choose k}(x/\sqrt{-s})^{n-k}\mu_k=H_n(x,s).
\end{align}
This is equivalent to say that the generating function  of the  Hermite polynomials $H_n(x,s)$ has the  following
continued fraction expansion:
\begin{align}
H(z,x,s)=\sum_{n\geq 0} H_n(x,s)z^n=
\cfrac{1}{1-xz+\cfrac{sz^2}{1-xz+\cfrac{2sz^2}{1-xz+\cfrac{3sz^2}{\ddots}}}}.
\end{align}

Several $q$-Hermite polynomials were introduced in the literature (see \cite{KK, ISV}).
Two important classes of orthogonal $q$-analogues of $H_n(x,s)$ are 
the continuous and the discrete $q$-Hermite I polynomials,  which are special cases of the Al-Salam-Chihara polynomials. 

We first introduce some standard $q$-notations.   For $n\geq 1$  let
$$
 [n]:=[n]_q=\frac{1-q^n}{1-q}, \quad [n]_q!=\prod_{k=1}^n[k]_q,\quad [2n-1]_q!!=\prod_{k=1}^n[2k-1]_q,
$$
and $(a;q)_n=(1-a)(1-aq)\cdots (1-aq^{n-1})$ with $(a;q)_0=1$. The $q$-binomial coefficient is defined by
$$
{n\brack k}:={n\brack k}_q=\frac{(q;q)_n}{(q;q)_k(q;q)_{n-k}}
$$
for  $0\leq k\leq n$ and zero otherwise. 

 Recall \cite{IS} that 
the Al-Salam-Chihara polynomials $P_n(x;a,b,c)$ satisfy the three term recurrence:
\begin{align}\label{eq:ACR}
P_{n+1}(x; a,b,c)=(x-aq^n)P_n(x;a,b,c)-(c+bq^{n-1})[n]_qP_{n-1}(x;a,b,c)
\end{align}
with initial values $P_{-1}(x;a,b,c)=0$ and $P_{0}(x;a,b,c)=1$.

\begin{defn} Let $\F_{a,b,c}$ be the linear functional on the polynomials in $z$ which satisfy
\begin{align}\label{eq:moment}
\F_{a,b,c}(P_n(z;a,b,c))=\delta_{n,0}.
\end{align}
\end{defn}

Note that the continuous $q$-Hermite polynomials are
\begin{align}\label{eq:contherm}
\tilde{H}_n(x,s|q)=P_n(x;0,0,s)
\end{align}
and are also the moments  (see \cite{IS} and Proposition~16):
\begin{align}\label{eq:contherm2}
\tilde{H}_n(x,s|q)=\F_{x,-s,0}(z^n).
\end{align}
The discrete $q$-Hermite polynomials  I are
\begin{align}\label{eq:disch}
\tilde{h}_n(x,s;q)=P_n(x;0,(1-q)s,0).
\end{align}
The discrete $q$-Hermite polynomials  II  are
\begin{align}\label{eq:disch}
\tilde{h}_n(x;q)=(-i)^n \tilde{h}_n(ix,1;q^{-1}).
\end{align}
Besides, as we will see in Section~4,  the polynomials
\begin{align}\label{eq:H2}
h_n(x,s;q):=P_n(0; -x,0, s)
\end{align}
 are actually a rescaled version of $\tilde{h}_n(x;q)$.
The main purpose of this paper is to study   another  $q$-analogue of Hermite polynomials.
\begin{defn} The new $q$-Hermite polynomials are defined by
\begin{align}\label{eq:H1}
H_n(x,s|q):=\F_{x,0,-s}(z^n).
\end{align}
\end{defn}

The q-Hermite polynomials $\tilde H_n(x,s|q)$ have, 
amongst other facts, 

\begin{itemize}
\item[(1)] orthogonality with an explicit measure, 
\item[(2)] an explicit 3-term recurrence relation, 
\item[(3)]  explicit expressions,
\item[(4)]  a combinatorial model using matchings, 
\item[(5)]  are moments for other orthogonal polynomials, 
\item[(6)]  evaluable Hankel determinants because of (5), 
\item[(7)] Jacobi continued fractions as generating  functions because of (5). 
\end{itemize}

The new $q$-Hermite polynomials  $H_n(x,s|q)$ are not orthogonal,  i.e., they do not have (1) and (2). 
Instead they have  a nice $q$-analogue of  the operator formula \eqref{eq:op} for the ordinary Hermite polynomials (see Theorem~\ref{thm:main}), 
  the coefficients in  $H_n(x,s|q)$  appear in the inverse matrix of the coefficients in the continuous 
  $q$-Hermite polynomials (cf. Theorem~\ref{thm:HH}),  they have simple connection coefficients  with  
   $q$-Lucas  and $q$-Fibonacci polynomials  (cf. Theorem~\ref{thm:hf}). 
  The discrete q-Hermite polynomials $h_n(x,s;q)$ also have (1)-(4), and we will show in Theorem~\ref{thm4.4}
   that they are also moments. 
Moreover,  the quotients of two consecutive  polynomials $h_n(x,s;q)$ (see Eq.\eqref{eq:coeff}) appear as coefficients 
 in the expansion of the S-continued fraction of the generating function of
 $H_n(x, s|q)$'s, which  leads to a second proof of   Theorem~\ref{thm:main}.
 
 This paper is organized as follows:  we prove the main properties of $H_n(x,s|q)$ and $h_n(x,s;q)$  in Section~2 and  Section~3, respectively. In Section~4 we shall establish the connection of our new $q$-Hermite polynomials with the $q$-Fibonacci and $q$-Lucas polynomials. This yields,  in particular,  a generalization of Touchard-Riordan's formula for the moments of continuous $q$-Hermite polynomials (cf. Proposition~\ref{prop:mathieu}), first obtained by Josuat-Verg\`es~\cite{JV}. Finally,  in Section~5, we recall some well-known facts about the general theory of orthogonal polynomials and give another proof of the generalized Touchard-Riordan formula by using the orthogonality of continuous $q$-Hermite polynomials.
\section{The  $q$-Hermite polynomials  $H_n(x,s|q)$}
By  \eqref{eq:moment}  the  $q$-Hermite polynomials $H_n(x,s|q)$ are  the moments of the measure of the orthogonal polynomials
$P_n(z)$ satisfying the recurrence:
\begin{align}\label{eq:defP}
P_{n+1}(z)=(z-xq^n) P_n(z)+s[n]_qP_{n-1}(z).
\end{align}
Recall \cite[p.80]{KK} that 
the Al-Salam-Chihara polynomials $Q_n(x):=Q_n(x;\alpha, \beta)$ satisfy the three term recurrence:
\begin{align}\label{eq:AC}
Q_{n+1}(x)=(2x-(\alpha+\beta)q^n)Q_n(x)-(1-q^n)(1-\alpha\beta q^{n-1})Q_{n-1}(x),
\end{align}
with $Q_0(x)=1$ and $Q_{-1}(x)=0$. They have the following explicit formulas:
\begin{align}
Q_n(x;\alpha,\beta|q)
&=(\alpha e^{i\theta};q)_n e^{-i\theta}{}_2\phi_1
\left(\begin{array}{cl}
q^{-n},&\beta e^{-i\theta}\\
\quad \alpha^{-1} q^{-n+1}e^{-i\theta}&
\end{array}|q;\alpha^{-1} q e^{i\theta}\right),\label{eq:Q2}
\end{align}
where $x=\cos \theta$.

Comparing \eqref{eq:defP} and \eqref{eq:AC}  we have 
$P_n(z)=\frac{1}{(2a)^n}Q_n(az; \alpha, 0)$ with 
\begin{align}
a=\frac{1}{2}\sqrt{\frac{q-1}{s}}\quad\textrm{and}\quad 
\alpha=x\sqrt{\frac{q-1}{s}}.
\end{align}

Using the known formula for Al-Salam-Chihara polynomials we obtain
\begin{align}
P_n(z)&=\frac{1}{(2a\alpha)^n}\sum_{k=0}^n\frac{(q^{-n};q)_k}{(q;q)_k}q^k
\prod_{i=0}^{k-1}(1+\alpha^2q^{2i}-2q^i a\alpha z)\nonumber\\
&=\left(\frac{s}{x(q-1)}\right)^n
\sum_{k=0}^n\frac{(q^{-n};q)_k}{(q;q)_k}\left(\frac{-q}{s}\right)^k
\prod_{i=0}^{k-1}\left((q-1)q^ixz-s-(q-1)q^{2i}x^2\right).\label{rescal}
\end{align}

The first values of these polynomials are
\begin{align*}
 P_1(z)&=z-x,\\
P_2(z)&=z^2-x(1+q)z+(s+qx^2),\\
P_3(z)&=z^3-x[3]_q z^2+(2s+qs+q[3]_q x^2)z
-(s+qs+q^2s+q^3x^2)x.
\end{align*}

A matching $m$ of $\{1,2,\ldots, n\}$ is a set of pairs $(i,j)$ such that $i<j$ and $i,j\in [n]$. 
Each pair $(i,j)$ is called an edge of the matching. Let $\ed(m)$ be the number of edges of $m$, so $n-2\ed(m)$ is the number of unmatched vertices. Two edges $(i,j)$ and $(k,l)$ have a crossing if $i<k<j<l$ or $k<i<l<j$. Let 
$\cro(m)$ be the number of crossing numbers in the matching $m$. 
Using the combinatorial theory of Viennot~\cite{Vi},  Ismail and  Stanton \cite[Theorem 6]{IS} gave a combinatorial interpretation of the moments of Al-Salam-Chihara polynomials.  In particular  we derive   the  following result from  \cite[Theorem 6]{IS}.
\begin{lem} The moments  of the measure of  the orthogonal polynomials  $\{P_n(x)\}$ 
 are  the generating functions for all matchings $m$ of 
$[n]$:
\begin{align}\label{eq:combmoment}
\F_{x,0,-s}(z^n)=\sum_{m}x^{n-2\ed(m)}(-s)^{\ed(m)}
q^{\c(m)+\cro(m)},
\end{align}
where  
$c(m)=\sum_{\mathrm{a-vertices}}
|\{\mathrm{edges}\; i<j:\; i<a<j\}|$  and the sum extends over all matchings $m$ of $[n]$.
\end{lem}

Let $M(n,k)$ be  the set of matchings of $\{1, \ldots, n\}$ with $k$ unmatched vertices. Then
\begin{align}\label{eq:cig11}
\F_{x,0,-s}(z^n)=\sum_{k}c(n,k,q)x^k(-s)^\frac{n-k}{2},
\end{align}
where
\begin{align}\label{eq:key}
c(n,k,q)=\sum_{m\in M(n,k)}
q^{c(m)+\cro(m)}.
\end{align}
It is easy to verify that
\begin{align}\label{eq:cig2}
c(n,k,q)=c(n-1,k-1,q)+[k+1]_q c(n-1,k+1,q)
\end{align}
with $c(0,k,q)=\delta_{k,0}$ and $c(n,0,q)=c(n-1,1,q)$.

Indeed, if $n$ is an unmatched vertex then for the restriction $m_0$ of $m$ to $[n-1]$ we get $c(m_0)=c(m)$
and $\cro(m_0)=\cro(m)$. If $n$ is matched with $m(n)$, such that there are $i$ unmatched vertices and $j$ endpoints of edges which cross the edge $(m(n),\;n)$ between $m(n)$ and $n$, then $c(m)=c(m_0)+i-j$ and 
$\cro(m)=\cro(m_0)+j$. Thus 
$c(m)+\cro(m)=c(m_0)+\cro(m_0)+i$. Since each $i$ with $0\leq i\leq k$ can occur we get \eqref{eq:cig2}.

Let now  ${\D}_q$ be the  $q$-derivative operator defined by
 $$
 {\D}_qf(z)=\frac{f(z)-f(qz)}{(1-q)z}.
 $$ 
 We have then the following $q$-analogue of \eqref{eq:op}.
\begin{thm}\label{thm:main} The $q$-Hermite polynomials $H_n(x,s|q)$, defined as 
moments $\F_{x,0,-s}(z^n)$, have the following operator formula:
\begin{align}\label{eq:qop}
H_n(x,s|q)=(x-s {\D}_q)^n1.
\end{align}
\end{thm}
\begin{proof}
We know that
\begin{align}\label{eq:cig1}
H_n(x,s|q)=\sum_{k}c(n,k,q)x^k(-s)^\frac{n-k}{2},
\end{align}
where $c(n,k,q)$ satisfies \eqref{eq:cig2}.   Therefore
\begin{align*}
H_n(x,s|q)&=\sum_kc(n-1,k-1,q)x^k(-s)^\frac{n-k}{2}+\sum_k[k+1]_qc(n-1,k+1,q)x^k(-s)^\frac{n-k}{2}\\
&=xH_{n-1}(x,s|q)-s{\D}_q H_{n-1}(x,s|q).
\end{align*}
The result follows then by induction on $n$.
\end{proof}

\noindent{\bf Remark.}  It should be noted that the method of  Varvak \cite{Va}  (see also \cite{JV}) can also be applied to prove Theorem~\ref{thm:main}. In fact her
 method proves first that $(x-sD_q)^n1$ is a generating function of some rook placements, which 
 is then shown to count  involutions  with respect to the statistic 
 ${\c(m)+\cro(m)}$ (see  \cite[Theorem 6.4]{Va}). We will give another proof of \eqref{eq:qop} by  using continued fraction, see 
the remark after Theorem~\ref{thm:Hh}. 

\medskip
The first terms  of the sequence $H_n(x,s|q)$ are 
\begin{align*}
&1,\; x,\; -s+x^2,\; x(-(2+q)s+x^2),\; (2+q)s^2-(3+2q+q^2)sx^2+x^4,
\\ 
&\qquad\qquad x((5+6q+3q^2+q^3)s^2-(4+3q+2q^2+q^3)sx^2+x^4),\ldots
\end{align*}

Let 
\begin{align}\label{eq:bcoeff}
\tilde H_n(x,s|q)=\sum_{k}b(n,k,q)x^k(-s)^\frac{n-k}{2}.
\end{align}
\begin{thm}\label{thm:HH} The matrices $(c(i,j,q))_{i,j=0}^{n-1}$ and $(b(i,j,q)(-1)^\frac{i-j}{2})_{i,j=0}^{n-1}$ are mutually inverse.
\end{thm}
\begin{proof}
We first show by induction that
\begin{align}\label{eq1.1}
\tilde H_n(x+s{\D}_q,s|q)1=x^n.
\end{align}
For this is obvious for $n=0$. If it is already shown for $n$ we get
\begin{align*}
\tilde H_{n+1}(x+s{\D}_q,s|q)1&=(x+s{\D}_q)\tilde H_n(x+s{\D}_q,s|q)1-s[n]_q\tilde H_{n-1}(x+s{\D}_q,s|q)1\\
&=(x+s{\D}_q)x^n-s[n]_qx^{n-1}=x^{n+1}.
\end{align*}
On the other hand we have
\begin{align}\label{eq1.2}
\tilde H_{n}(x+s{\D}_q,s|q)1&=\sum_{k=0}^nb(n,k,q)(-s)^\frac{n-k}{2}(x+s{\D}_q)^k1\nonumber\\
&=\sum_{k=0}^nb(n,k,q)(-s)^\frac{n-k}{2}\sum_{j=0}^kc(k,j,q)s^\frac{k-j}{2}x^j\nonumber\\
&=\sum_{j=0}^n s^\frac{n-j}{2}x^j\sum_{k=j}^n b(n,k,q)(-1)^\frac{n-k}{2}c(k,j,q).
\end{align}
The result follows then by comparing \eqref{eq1.1} and \eqref{eq1.2}.
\end{proof}

\noindent
{\bf Remark.}
If we set $q=0$ then \eqref{eq:cig2} reduces to the well-known Catalan triangle (see \cite[Chap. 7]{aigner}), which implies 
\begin{align*}
 c(2n,0,0)&=C_n=\frac{1}{n+1}{2n\choose n},\\
c(2n,2k,0)&=\frac{2k+1}{n+k+1}{2n\choose n-k}={2n\choose n-k}-{2n\choose n-k-1},\\
c(2n+1,2k+1,0)&=\frac{2k+2}{n+k+2}{2n+1\choose n-k}={2n+1\choose n-k}-{2n+1\choose n-k-1}.
\end{align*}




We shall use the standard $q$-notations in \cite{GR}.   
The recurrence \eqref{eq:defP} implies that the  Hankel determinants of $H_n(x,s|q)$ are
\begin{align}
\det(H_{i+j}(x,s|q))_{i,j}^{n-1}=(-s)^{n\choose 2}\prod_{j=0}^{n-1}[j]_q! 
\end{align}
and 
\begin{align}
\det(H_{i+j+1}(x,s;q))_{i,j}^{n-1}
=h_n(x,-s;q)(-s)^{n\choose 2}\prod_{j=0}^{n-1}[j]_q! ,
\end{align}
where
$$
h_n(x, -s;q)=(-1)^nP_n(0)=\left(\frac{s}{x(1-q)}\right)^n\sum_{k=0}^n\frac{(q^{-n};q)_k}{(q;q)_k}q^k
\prod_{i=0}^{k-1}(1+x^2(q-1)q^{2i}/s).
$$

\section{The rescaled discrete  $q$-Hermite polynomials  II}
By definition \eqref{eq:H2} and \eqref{eq:ACR} we have 
\begin{align}
h_{n+1}(x,s;q)=q^nx h_n(x,s;q)-[n]_q s h_{n-1}(x,s;q).
\end{align}
We derive then
\begin{align}
h_n(x,s;q)&=q^{n\choose 2}\sqrt{s^n} \tilde h_n\left(\frac{x}{\sqrt{s}};q\right)\\
&=\sum_{k=0}^nq^{n-2k\choose 2}{n\brack 2k}[2k-1]_q!! (-s)^k x^{n-2k},
\end{align}
where $\tilde h_n(x;q)$ are the discrete $q$-Hermite polynomials II, see \eqref{eq:H2}.

Since ${\D}_q(fg)={\D}_q(f)g+f(qx){\D}_q(g)$ and ${\D}_q(x)=1$, we see that
$$
{\D}_q(h_{n+1}(x))=q^nx{\D}_q(h_n(x))+q^nh_n(qx)-[n]_qs{\D}_q(h_{n-1}(x)).
$$
We derive by induction on $n$ that
\begin{align}\label{eq:qdif}
{\D}_q h_n(x,s;q)=[n]_q h_{n-1}(qx,s;q).
\end{align}

The first terms of these polynomials are
$$
1,\quad x,\quad qx^2-s,\quad q^3x^3-s[3]_qx,\quad 
q^6x^4-s(q^5+q^4+2q^3+q^2+q)x^2+s^2[3]_q. 
$$
The following result shows that  the polynomials $h_n(x,s;q)$ are moments of some orthogonal polynomials.
\begin{thm}\label{thm4.4}
The generating function of  $h_n(x,s;q)$  has the  continued fraction expansion:
 $$
 \sum_{m\geq 0}h_n(x,s;q)t^n=\cfrac{1}{ 1 -b_0 t-
\cfrac{\la_1 t^2}{1 -b_1 t-
 \cfrac{\la_2t^2}{1 - b_2 t-
 \cfrac{\la_3t^2}{1 -
 \ddots}}}},
 $$
 with
\begin{align}\label{eq:parameter}
b_n=q^{n-1}(q^n+q^{n+1}-1)x\quad \text{and}\quad \la_n=-q^{n-1}[n]_q(s+q^{2n-2}(1-q)x^2).
\end{align}
\end{thm}
\begin{proof}
To prove this it suffices to show that the Stieltjes tableau \eqref{eq:stableau} is satisfied with 
$$
a(n,k)={n\brack k}h_{n-k}(q^kx,s;q).
$$
This is easily verified.
\end{proof}

This implies that their Hankel determinants are
\begin{align}
\det(h_{i+j}(x,s;q))_{i,j}^{n-1}=(-1)^{n\choose 2}q^{n\choose 3}\prod_{j=0}^{n-1}\left([j]_q! (s+q^{2j}(1-q)x^2)^{n-1-j}\right)
\end{align}
and 
\begin{align}
\frac{\det(h_{i+j+1}(x,s;q))_{i,j}^{n-1}}
{\det(h_{i+j}(x,s;q))_{i,j}^{n-1}}
=w(n),
\end{align}
where  $w(n)$ satisfies
$$
w(n+1)=q^{n-1}(q^n+q^{n+1}-1)xw(n)+q^{n-1}[n]_q (s+q^{2n-2}(1-q)x^2)w(n-1).
$$
It is easily verified that
\begin{align}
 w(n)=\sum_{k=0}^nq^{2{n-k\choose 2}}{n\brack 2k}[2k-1]_q!! s^k x^{n-2k}
\end{align}
satisfies the same recurrence with the same initial values.


\begin{lem} Let $L_n(x):=h_{n}(x,(1-q)s;q)$. Then 
\begin{align}\label{eq:basic}
s L_n(x)+xL_{n+1}(x)=(x^2+s)L_n(qx).
\end{align}
\end{lem}
\begin{proof} First we note that the constant terms of both sides of \eqref{eq:basic} are equal to $sL_n(0)$. So it suffices to 
show that the derivatives of the two sides are equal.  
Applying ${\D}_q$ to \eqref{eq:basic} and using \eqref{eq:qdif}  we obtain, after replacing $x$ by $x/q$,
$$
s[n]L_{n-1}(x)+xq[n-1]L_n(x)+L_{n+1}(x)=(x^2+s)q[n]L_{n-1}(qx).
$$
Since $L_{n+1}(x)=q^nxL_n(x)-(1-q^n)sL_{n-1}(x)$, we can rewrite the above equation as follows:
\begin{align}
sL_{n-1}(x)+xL_n(x)=(x^2+s)L_{n-1}(qx).
\end{align}
The proof is thus completed by induction on $n$.
\end{proof}

We shall prove the following Jacobi  continued fraction expansion for the 
generating function of $(x+(1-q)s{\D}_q)^n\cdot 1$. This is equivalent to Theorem~\ref{thm:main}.
\begin{thm} \label{thm:Hh} Let $T_n(x,s)=(x+(1-q)s{\D}_q)^n\cdot 1$. Then
\begin{align}
\sum_{n\geq 0} T_n(x,s) t^n=
\cfrac{1}{1-b_0 t-
\cfrac{\lambda_1 t^2}{1-b_1 t-
\cfrac{\lambda_2 t^2}{1-\ddots}}},
\end{align}
where the coefficients are 
\begin{equation}\label{eq:coeff} 
b_n=q^n x,\quad \text{for $n\geq 0$};\quad\text{and}\quad  
\lambda_n=(1-q^n) s,\quad \text{for $n\geq 1$}.
\end{equation}
\end{thm}
\begin{proof}
Since $T_n(x,s)=(x+(1-q)s{\D}_q) T_{n-1}(x,s)$, we have
$$
T_n(x,s)=(x+\frac{s}{x})T_{n-1}(x,s)-\frac{s}{x}T_{n-1}(qx,s).
$$
Equivalently the generating function
$G(x,t)=\sum_{n\geq 0}T_n(x,s)t^n$ satisfies
the functional equation:
\begin{equation}\label{func-eq}
\left(1-\frac{x^2+s}{x}t\right)G(x,t)=1-\frac{s}{x}tG(qx,t).
\end{equation}
Suppose that
\begin{align}\label{eq:cfstieltjes}
G(x,t)=\cfrac{1}{ 1 -
\cfrac{c_1t}{1 -
 \cfrac{c_2t}{1 - 
 \cfrac{c_3t}{1 -
 \ddots}}}},
 \end{align}
where  $c_n=(g_n-1)g_{n-1}A$ with $A:=A(x)=-\frac{x^2+s}{x}$ and $g_i:=g_i(x)$. 

Substituting \eqref{eq:cfstieltjes} in \eqref{func-eq}  and then replacing $t$ by $t/A$  we obtain
\begin{align}
{1+t\over \displaystyle 1 -{(g_1-1)t\over
\displaystyle 1 - {(g_2-1)g_1t\over
\displaystyle 1 - {(g_3-1)g_2t\over
\displaystyle 1 - {(g_4-1)g_3t\over
\displaystyle 1 - \ldots}}}}}
=1 +{\frac{s}{x^2+s}t\over
\displaystyle 1 - {(g_1'-1)\frac{A'}{A}t\over
\displaystyle 1 - {(g_2'-1)g_1'\frac{A'}{A}t\over
\displaystyle 1 - {(g_3'-1)g_2'\frac{A'}{A}t\over
\displaystyle 1 - \ldots}}}},
\end{align}
where $A':=A(qx)$ and $g_i':=g_i(qx)$.
Comparing this with  Wall's formula (see \cite{KZ}):
\begin{align}
{1+z\over \displaystyle 1 -{(g_1-1)z\over
\displaystyle 1 - {(g_2-1)g_1z\over
\displaystyle 1 - {(g_3-1)g_2z\over
\displaystyle 1 - {(g_4-1)g_3z\over
\displaystyle 1 - \ldots}}}}}
=1 +{g_1z\over
\displaystyle 1 - {(g_1-1)g_2z\over
\displaystyle 1 - {(g_2-1)g_3z\over
\displaystyle 1 - {(g_3-1)g_4z\over
\displaystyle 1 - \ldots}}}},
\end{align}
we derive that $g_0=1$ and for $n\geq 1$, 
\begin{equation}\label{eq1}\left\{\begin{split}
g_{2n}&=\frac{A'}{A}\,\frac{g_{2n-1}'-1}{g_{2n-1}-1}\,g_{2n-2}',\\
g_{2n+1}&=\frac{A'}{A}\,\frac{g_{2n}'-1}{g_{2n}-1}\,g_{2n-1}'.
\end{split}\right.
\end{equation}

For example,
\begin{align*}
g_1&=\frac{s}{x^2+s},\quad 
&g_3&=\frac{A'}{A}\frac{g_2'-1}{g_2-1}g_1'=\frac{s}{x^2+s}\frac{1}{q},\\
g_2&=\frac{A'}{A}\frac{g_1'-1}{g_1-1}=q,\quad
&g_4&=\frac{A'}{A}\frac{g_3'-1}{g_3-1}g_2'={\frac {-s+qs+{q}^{3}{x}^{2}}{-s+qs+q{x}^{2}}}.
\end{align*}
In general we have the following result.
\begin{equation}\label{keystep}\left\{\begin{split}
g_{2n}&=\frac{sL_n(x)+xL_{n+1}(x)}{(x^2+s)L_{n}(x)},\\
g_{2n+1}&=\frac{s L_{n}(x)} {sL_{n}(x)+xL_{n+1}(x)}.
\end{split}\right.\qquad (n\geq 0).
\end{equation}
This can be verified by induction on $n$.  Suppose that the formula~\eqref{keystep}  is true for $n\geq 0$. 
We prove that the formula holds for $n+1$. By \eqref{eq1} we have
\begin{align*}
g_{2n+2}=\frac{A'}{A}\,\frac{g_{2n+1}'-1}{g_{2n+1}-1}\,g_{2n}'
=\frac{s L_n(x)+xL_{n+1}(x)}{(x^2+s)L_{n+1}(x)}\, \frac{L_{n+1}(qx)}{L_n(qx)}.
\end{align*}
It follows from Lemma 1 that 
\begin{align}\label{eq:induction}
g_{2n+2}=\frac{s L_{n+1}(x)+xL_{n+2}(x)}{(x^2+s)L_{n+1}(x)}.
\end{align}
Since 
\begin{align}\label{eq:l1}
L_{n+1}(x)-xL_n(x)=(q^n-1)(xL_n(x)+sL_{n-1}(x)),
\end{align}
the  verification for   $g_{2n+3}$ is then straightforward.
We  derive  from \eqref{eq:cfstieltjes} and \eqref{keystep} that
\begin{equation}\label{eq:coeff} 
\left\{\begin{split}
c_{2n}&=(g_{2n}-1)g_{2n-1}A=(1-q^n)s\frac{L_{n-1}(x)}{L_n(x)},\quad \text{for $n\geq 1$};\\
c_{2n+1}&=(g_{2n+1}-1)g_{2n}A=\frac{L_{n+1}(x)}{L_n(x)},\quad \text{for $n\geq 0$}.
\end{split}\right.
\end{equation}
Invoking the \emph{contraction formula} (see \cite{zeng}), which transforms 
a S-continued fraction to a J-continued fraction, 
\begin{align}\label{eq:contraction}
 {1\over \displaystyle 1 -{c_1z\over
\displaystyle 1 - {c_2z\over
\displaystyle 1 - {c_{3}z\over
\displaystyle 1-{c_4z\over
\displaystyle  \ddots
}}}}}=\cfrac{1}{1-c_1z-\cfrac{c_1c_2 z^2}{
 1 -(c_2+c_3)z-\cfrac{c_3c_4z^2}{\ddots}}},
\end{align}
we obtain
\begin{equation}\label{eq:jacobi}
\left\{\begin{split}
 b_n&=\frac{h_{n+1}(x,(1-q)s;q)}{h_n(x,(1-q)s;q)}+(1-q^n)s\frac{h_{n-1}(x,(1-q)s;q)}{h_n(x,(1-q)s;q)}=q^nx,\\
\la_n&=\frac{h_{n}(x, (1-q)s;;q)}{h_{n-1}(x, (1-q)s;q)}\cdot (1-q^n)s\frac{h_{n-1}(x, (1-q)s;;q)}{h_n(x, (1-q)s;;q)}=(1-q^n) s.
\end{split}\right.
\end{equation}
This completes the proof. 
\end{proof}

\noindent{\bf Remark.}  Instead of the contraction formula~\eqref{eq:contraction}, we can also proceed as follows. Define a table $(A(n,k))_{n,k\geq 0}$ by
\begin{align}
 A(0,k)&=\delta_{k,0},\nonumber\\
A(n,0)&=c_1 A(n-1,1),\label{ctable}\\
A(n,k)&=A(n-1,k-1)+c_{k+1}A(n-1,k+1).\nonumber
\end{align}

In this case $A(2n,2k+1)=A(2n+1,2k)=0$ for all $n,k$.
If we define 
$$
a(n,k)=A(2n,2k),
$$
then it is easily verified that $a(n,k)$ satisfy \eqref{eq:stableau} with
\begin{align}
b_0=c_1,\quad b_n=c_{2n}+c_{2n+1},\quad \la_n=c_{2n}c_{2n-1}.
\end{align}
Substituting the values in \eqref{eq:coeff} for $c_n$ 
we obtain \eqref{eq:jacobi}.
Therefore 
$$\sum_{n}A(2n,0)t^n=\sum_n a(n,0)t^n=\sum_n T_n(x,s)t^n.
$$

As another application of this remark  we prove the following result.

\begin{prop}
 Let $w_n(m,q)=q^\frac{n((2m+1)n+1)}{2}$. Then
 $$
 \sum_{m\geq 0}w_n(m,q)t^n=\cfrac{1}{ 1 -b_0 t-
\cfrac{\la_1 t^2}{1 -b_1 t-
 \cfrac{\la_2t^2}{1 - b_2 t-
 \cfrac{\la_3t^2}{1 -
 \ddots}}}},
 $$
 where
\begin{align*}
b_n&=q^{(2m+1)n-m}(q^{(2m+1)n}-1)+q^{(2m+1)(2n+1)-m},\\
\la_n&=q^{(2m+1)(3n-1)-2m}(q^{(2m+1)n}-1).
\end{align*}
\end{prop}
\begin{proof}
 Let 
$$
A(2n,2k)=\frac{w_n(m,q)}{w_k(m,q)}{n\brack k}_{q^{2m+1}}\quad
\text{and}\quad A(2n+1,2k+1)=\frac{w_{n+1}(m,q)}{w_{k+1}(m,q)}{n\brack k}_{q^{2m+1}}.
$$
Then it is easily verified that the table \eqref{ctable} holds with 
$c_{2n}=q^{(2m+1)n-m}(q^{(2m+1)n}-1)$ and $c_{2n+1}=q^{(2m+1)(2n+1)-m}$.
Therefore 
$$
\sum_{n}A(2n,0)t^n=\sum_n a(n,0)t^n=\sum_n w_n(m,q)t^n.
$$
\end{proof}

\section{Connection with $q$-Fibonacci polynomials and $q$-Lucas polynomials}

We define the Lucas polynomials  by
$$
l_n(x,s)=xl_{n-1}(x,s)+sl_{n-2}(x,s)\quad \text{for}�\quad n>2,
$$
with initial values $l_1(x,s)=x$ and $l_2(x,s)=x^2+2s$. They have the explicit formula
\begin{align}\label{eq:luc1}
l_n(x,s)=\sum_{2k\leq n}
\frac{n}{n-k} {n-k\choose k}s^k x^{n-2k}  \quad (n>0).
\end{align}
Furthermore we define $l_0(x,s)=1$. Note that this definition differs from the usual one in which 
$l_0(x,s)=2$. 

The Fibonacci polynomials are defined by
$$
f_n(x,s)=xf_{n-1}(x,s)+sf_{n-2}(x,s)
$$
with $f_0(x,s)=0$ and $f_1(x,s)=1$.  They have the explicit formula
\begin{align}\label{eq:fib1}
f_n(x,s)=\sum_{k=0}^{\left\lfloor\frac{n-1}{2}\right\rfloor} {n-1-k\choose k}s^k x^{n-1-2k}.
\end{align}

We first establish the following inversion of  \eqref{eq:luc1} and \eqref{eq:fib1}.
\begin{lem}
\begin{align}\label{eq:lem}
x^n&=\sum_{2k\leq n}{n\choose k} s^kl_{n-2k}(x,-s),\\
x^n&=\sum_{2k\leq n+1}\left({n\choose k}-{n\choose k-1}\right)s^k f_{n+1-2k}(x,-s).\label{eq:lemfib}
\end{align}
\end{lem}
\begin{proof}
Recall the Tchebyshev inverse relations~\cite[p. 54-62]{Ri}:
\begin{align}\label{eq:tch1}
b_n=\sum_{k=0}^{\left\lfloor\frac{n}{2}\right\rfloor} 
(-1)^k\frac{n}{n-k}{n-k\choose k}a_{n-2k}
\Longleftrightarrow
a_n=\sum_{k=0}^{\left\lfloor\frac{n}{2}\right\rfloor}
{n\choose k}
b_{n-2k},
\end{align}
where $a_0=b_0=1$, and 
\begin{align}\label{eq:tch}
b_n=\sum_{k=0}^{\left\lfloor\frac{n}{2}\right\rfloor} 
(-1)^k{n-k\choose k}a_{n-2k}
\Longleftrightarrow
a_n=\sum_{k=0}^{\left\lfloor\frac{n}{2}\right\rfloor}
\left[{n\choose k}-{n\choose k-1}\right]
b_{n-2k}.
\end{align}
We derive immediately \eqref{eq:lem} from \eqref{eq:luc1} and \eqref{eq:tch1}.
Clearly \eqref{eq:fib1} is equivalent to the left identity in \eqref{eq:tch} with $a_n=\left(\frac{x}{\sqrt{s}}\right)^n$ and $b_n=\frac{f_{n+1}(x,-s)}{(\sqrt{s})^n}$. By inversion we find
\begin{align}\label{eq:lem'}
x^n=\sum_{k=0}^{\left\lfloor\frac{n}{2}\right\rfloor}
\left({n\choose k}-{n\choose k-1}\right)
s^k f_{n+1-2k}(x,-s).
\end{align}
To see the equivalence of \eqref{eq:lemfib} and \eqref{eq:lem'} we notice that 
\begin{itemize}
\item if $n$ is odd, then
${n\choose k}={n\choose k-1}$ for $k=\left\lfloor\frac{n+1}{2}\right\rfloor$,
\item if $n$ is even, then $\left\lfloor\frac{n+1}{2}\right\rfloor=\left\lfloor\frac{n}{2}\right\rfloor$.
\end{itemize}
\end{proof}

\medskip
Define the $q$-Lucas and $q$-Fibonacci polynomials by
\begin{align}
L_n(x,s)&=l_n(x+(q-1)s\D_q,s)\cdot 1,\\
F_n(x,s)&=f_n(x+(q-1)s{\D}_q, s)\cdot 1. 
\end{align}
It is known (see \cite{Cigler03} and \cite{Cigler09} ) that  they have the explicit formulae
\begin{align}
L_n(x,s)&=\sum_{k=0}^{\left\lfloor \frac{n}{2}\right\rfloor}q^{k\choose 2} \frac{[n]}{[n-k]}{n-k\brack k}s^kx^{n-2k},\label{eq:qluc}\\
F_n(x,s)&=\sum_{k=0}^{\left\lfloor \frac{n-1}{2}\right\rfloor}q^{k+1\choose 2}{n-1-k\brack k}s^kx^{n-1-2k},\label{eq:qfib}
\end{align}
for $n>0$,  with $L_0(x,s)=1$ and  $F_0(x,s)=0$.
\begin{thm}\label{thm:hf}
We have 
\begin{align}\label{eq:explicit}
H_n(x, (q-1)s|q)
&=\sum_{k=0}^{\left\lfloor \frac{n}{2}\right\rfloor}
{n\choose  k}s^kL_{n-2k}(x,-s)\\
&=\sum_{k=0}^{\left\lfloor \frac{n+1}{2}\right\rfloor}\left({n\choose k}-{n\choose k-1}\right) s^k F_{n+1-2k}(x,-s).\label{eq:explicit2}
\end{align}
\end{thm}
\begin{proof}
Since
\begin{align*}
L_n(x,-s)&=l_n(x-(q-1)s{\D}_q,s)\cdot 1,\\
F_n(x,-s)&=f_n(x-(q-1)s{\D}_q,s)\cdot 1,
\end{align*}
the theorem follows by applying the homomorphism $x\mapsto x-(q-1)s\D_q$ to
\eqref{eq:lem} and \eqref{eq:lemfib}.
\end{proof}

We derive some consequences of  the formula \eqref{eq:explicit2}. 
 \begin{cor} We have 
\begin{align}\label{ex1}
\begin{split}
H_n(1,q-1|q)&=\sum_{k=-n}^{n}(-1)^k
q^{\frac{k(3k+1)}{2}}       {n\choose \left\lfloor\frac{n-3k}{2}\right\rfloor}\\
&=\sum_{k=0}^{n}(-1)^k q^{\frac{k(3k+1)}{2}} 
{n\choose \left\lfloor\frac{n-3k}{2}\right\rfloor}+
\sum_{k=1}^{n}(-1)^kq^{\frac{k(3k-1)}{2}}{n\choose \left\lfloor\frac{n-3k+1}{2}\right\rfloor}.
\end{split}
\end{align}
\end{cor}
\begin{proof}
Let $r(k)=\frac{k(3k+1)}{2}$. Then, it  follows from \cite{Cigler03} that 
$$
F_{3n}(1,-1)=\sum_{k=-n}^{n-1}(-1)^k q^{r(k)},\quad F_{3n+1}(1,-1)=F_{3n+2}(1,-1)=\sum_{k=-n}^{n}(-1)^k q^{r(k)},
$$
or 
\begin{align}
F_{n}(1,-1)=\sum_{-n\leq 3j\leq n-1}(-1)^j q^{\frac{j(3j+1)}{2}}.
\end{align}

Let $w(n)=\sum_{k=0}^{ \left\lfloor\frac{n+1}{2}\right\rfloor}
\left({n\choose k}-{n\choose k-1}\right)F_{n+1-2k}(1,-1) $. Consider a fixed term $(-1)^jq^{r(j)}$. This term occurs in $F_n(1,-1)$ if $-\frac{n}{3}\leq j\leq \frac{n-1}{3}$.
We are looking for all $k$, such that this term occurs in 
$F_{n+1-2k}(1,-1)$. For $j\geq 0$ the largest such number is
$k_0= \left\lfloor\frac{n-3j}{2}\right\rfloor$. For $j\leq \frac{n-2k}{3}$ is equivalent with $k\leq k_0$. Therefore the coefficient of
$(-1)^jq^{r(j)}$ in $w(n)$ is $\sum_{k=0}^{k_0}\left({n\choose k}-{n\choose k-1}\right)={n\choose k_0}$. For $j<0$ we have 
$-\frac{n+1-2k}{3}\leq j$ is equivalent with $k\leq \lfloor \frac{n+1-3j}{2}\rfloor$. This gives the last sum in \eqref{ex1}.
\end{proof}
\begin{cor}
We have
\begin{align}\label{eq:H2n}
H_{2n}\left(1, \frac{q-1}{q} |q \right)=q^{-n} \sum_{j=-n}^n 
\left({2n\choose n-3j}-{2n\choose n-3j-1}\right)q^{2j(3j+1)},
\end{align}
and 
\begin{align}\label{eq:H2n1}
H_{2n+1}(1,\frac{q-1}{q}|q )=q^{-n}\sum_{j=-n}^n\left({2n+1\choose n-3j}-{2n+1\choose n-3j-1}\right)q^{2j(3j+2)}.
\end{align}
\end{cor}
\begin{proof}
Note that 
\begin{align}
H_{2n}(1, \frac{q-1}{q}|q)&=\frac{1}{q^n}\sum_{k=0}^n \left({2n\choose n-k}-{2n\choose n-k-1}\right)q^{k}F_{2k+1}(1,-\frac{1}{q}),\label{eq:step1}\\
H_{2n+1}(1, \frac{q-1}{q}|q)&=\frac{1}{q^n}\sum_{k=0}^{n+1} \left({2n+1\choose n+1-k}-{2n+1\choose n-k}\right)q^{k-1}F_{2k}(1,-\frac{1}{q}).\label{eq:step2}
\end{align}
Recall  (see \cite{Cigler03}) that
\begin{align}
F_{3n}(1, -\frac{1}{q})=0, \quad F_{3n+1}(1, -\frac{1}{q})=(-1)^n q^{r(n)},\quad F_{3n+2}(1, -\frac{1}{q})=(-1)^nq^{r(-n)}.
\end{align}
Hence
\begin{itemize}
\item if $k=3j$ then $2k+1=6j+1$ and $q^kF_{2k+1}(1,-\frac{1}{q})=q^{3j}F_{6j+1}(1,-\frac{1}{q})=q^{2j(3j+1)}$.
\item if $k=3j+1$ then $2k+1=6j+3$ and $q^{k}F_{2k+1}(1,-\frac{1}{q})=0$.
\item If $k=3j+2$ then $2k+1=6j+5$ and $q^{k}F_{2k+1}(1,-\frac{1}{q})=q^{3j+2}$.
\item if $k=3j$ then $2k=6j$ and $q^{k-1}F_{2k}(1,-\frac{1}{q})=0$.
\item if $k=3j+1$ then $2k=6j+2$ and $q^{k-1}F_{2k}(1,-\frac{1}{q})=q^{2j(3j+2)}$.
\item If .$k=3j+2$ then $2k=6j+4$ and $q^{k-1}F_{2k}(1,-\frac{1}{q})=-q^{(3j+1)(2j+2)}$.
\end{itemize}
Substituting the above values into \eqref{eq:step1} and \eqref{eq:step2} yields \eqref{eq:H2n} and \eqref{eq:H2n1}.
\end{proof}

Finally, from \eqref{eq:explicit} and  \eqref{eq:key} we derive two explicit formulae for the coefficient $c(n,k,q)$.
\begin{prop}\label{prop:mathieu}
 If $k\equiv n\pmod{2}$ then
\begin{align}
c(n,k,q)&=\sum_{m\in M(n,k)}q^{c(m)+\cro(m)} \nonumber\\
&=(1-q)^{-\frac{n-k}{2}}\sum_{j\geq 0}{n\choose \frac{n-k-2j}{2}} (-1)^j  
q^{j\choose 2} \frac{[k+2j]}{[k+j]}{k+j\brack j}\label{eq:lucasformula}\\
&=(1-q)^{-\frac{n-k}{2}}\sum_{j\geq 0}
\left({n\choose \frac{n-k-2j}{2}}-{n\choose \frac{n-k-2j-2}{2}}\right)
(-1)^j q^{{j+1\choose 2}}{k+j\brack k}.\label{eq:cn}
\end{align}
\end{prop}
In the next section we shall give another proof by using the orthogonality of the continuous $q$-Hermite polynomials.
Some remarks about the above formula are in order.
\begin{itemize}
\item[(a)] Formula  \eqref{eq:cn} has been obtained with different means by Josuat-Verg\`es \cite[Proposition 12]{JV} and is also used in \cite{CJPR}. It is easy to see that \eqref{eq:lucasformula} and \eqref{eq:cn} are equal by writing
$$
 \frac{[k+2j]}{[k+j]}= q^j+\frac{[j]}{[k+j]}.
$$
\item[(b)] When $k=0$, 
we recover 
a formula of Touchard-Riordan (see \cite{aigner,ISV,Pe}):
\begin{align}\label{eq:TR}
c(2n,0,q)=\sum_{m\in M(2n,0)} q^{\cro(m)}=\frac{1}{(1-q)^n}\sum_{j=-n}^n{2n\choose n+j}(-1)^jq^{j\choose 2}.
\end{align}
\item[(c)] Notice that  $H_{2n}(0,-1|q)=c(2n,0,q)$ and $H_{2n+1}(0,-1|q)=c(2n+1,0,q)=0$. Hence 
$$
\sum_{n\geq 0}c(n,0,q)t^n=\cfrac{1}{1-\cfrac{t^2}{1-\cfrac{[2]_q t^2}{1-\cfrac{[3]_qt^2}{1-\cdots}}}}.
$$
We derive then a known result
(see \cite{ISV}): the coefficient $c(n,0,q)$  coincides with the $n$-th  moment of the continuous $q$-Hermite polynomials $\tilde H(x,1|q)$, i.e., 
$$
\F(z^n)=c(n,0,q),
$$
where $\F$ is the linear functional on the polynomials in $z$ defined by $\F(\tilde H_n(x,1|q))=\delta_{n,0}$.  
\end{itemize}

As in \cite{KSZ} we can derive another double sum expression for $H_n(x,s|q)$. We omit the proof.
 
\begin{prop} We have
\begin{align}
H_n(x,s|q)&=
\sum_{k=0}^n(-1)^kq^{-{k\choose 2}}\sum_{i=0}^{k}
\left(\frac{s}{x(q-1)}q^{-i}+xq^i\right)^{n}\nonumber\\
&\qquad\qquad \times \prod_{j=0, j\neq i}^{k}\frac{1}{q^{-i}-q^{-j}+x^2\frac{q-1}{s}(q^i-q^j)}.\label{eq:explicitbis}
\end{align}
\end{prop}

\section{Appendix}
\subsection{Some well-known facts}
In this section we recall some well-known facts about orthogonal polynomials (see \cite{aigner, wall,Vi}).
Let $p_n(x)$ be a sequence of polynomials which satisfies the three term recurrence relation
\begin{align}
p_{n+1}(x)=(x-b_n)p_n(x)-\la_n p_{n-1}(x)
\end{align}
with initial values $p_0(x)=1$ and $p_{-1}(x)=0$.

Define the coefficients $a(n,k)$ ($0\leq k\leq n$) by
\begin{align}
\sum_{k=0}^n a(n,k)p_k(x)=x^n.
\end{align}
These are characterized by the {\em Stieltjes tableau}:
\begin{equation}\label{eq:stableau}\begin{split}
a(0,k)&=\delta_{k,0},\\
a(n,0)&=b_0a(n-1,0)+\la_1 a(n-1,1),\\
a(n,k)&=a(n-1,k-1)+b_k a(n-1,k)+\la_{k+1}a(n-1,k+1).
\end{split}
\end{equation}
If $\F$ is the linear functional such that $\F(p_n(x))=\delta_{n,0}$, then 
\begin{align}
\F(x^n)=a(n,0).
\end{align}
The generating function of the moments has the continued fraction expansion
\begin{align}\label{eq:cfrac}
\sum_{n\geq 0}\F(x^n)z^n=\cfrac{1}{1-b_0 z-\cfrac{\la_1 z^2}{1-b_1z-\cfrac{\la_2z^2}{1-\cdots}}}.
\end{align}
The Hankel determinants for the moments are
\begin{align}
d(n,0)=\det(\F(z^{i+j}))_{i,j=0}^{n-1}=\prod_{i=1}^{n-1}\prod_{k=1}^i \la_k,
\end{align}
and 
\begin{align}
d(n,1)=\det(\F(z^{i+j+1}))_{i,j=0}^{n-1}=d(n,0) (-1)^n p_n(0).
\end{align}

As an example we want to give another  simple proof of  \eqref{eq:contherm2}.
 \begin{prop} The continuous $q$-Hermite polynomials $\tilde{H}_n(x,s|q)$ defined by  \eqref{eq:contherm}, i.e., 
 \begin{align}\label{eq:contH}
 \tilde{H}_{n+1}(x,s|q)=x\tilde{H}_n(x,s|q)-s[n]_q\tilde{H}_{n-1}(x,s|q),
 \end{align}
 are the moments of the measure of 
  the orthogonal polynomials $p_n(z):=P_n(z;x,-s,0)$ defined by  the recurrence
\begin{align}
p_{n+1}(z)=(z-xq^n)p_n(z)+sq^{n-1}[n]_q  p_{n-1}(z).
\end{align}
\end{prop}
\begin{proof}
Let $b_n=q^nx$ and $\la_{n+1}=(-s)q^{n}[n+1]_q$ for $ n\geq 0$.
It is sufficient to verify that in this case  \eqref{eq:stableau} is satisfied with
\begin{align}
a(n,k)={n\brack k}\tilde H_{n-k}(z,s|q).
\end{align}
This is clearly equivalent to  \eqref{eq:contH}.
\end{proof}
We derive immediately the  Hankel determinants 
\begin{align}
d(n,0)=(-s)^{n\choose 2} q^{n\choose 3}\prod_{j=0}^{n-1}[j]_q!,
\end{align}
and 
\begin{align}
d(n,1)=d(n,0)r(n),
\end{align}
where $r(n)=(-1)^{n} p_n(0;x,-s,0)$.

Note that the polynomials $r(n)$ satisfy 
$$
r(n)=q^{n-1}x r(n-1)+q^{n-2}s[n-1]_q r(n-2).
$$
This implies that 
\begin{align}
r(n)=q^{\frac{n(n-2)}{2}}\tilde H_{n}\left(x\sqrt{q},-s|\frac{1}{q}\right).
\end{align}
The first terms of the sequence $\tilde H_{n}(x,s|q)$ are 
\begin{gather*}
1,\; x,\; -s+x^2,\; x(-(2+q)s+x^2),\; (1+q+q^2)s^2-(3+2q+q^2)sx^2+x^4,\\
x((3+4q+4q^2+3q^3+q^4)s^2-(4+3q+2q^2+q^3)sx^2+x^4).
\end{gather*}
From their recurrence relation we see that
$$
\tilde H_{2n}(0,s|q)=(-s)^n[2n-1]_q!!\quad \text{and}\quad \tilde H_{2n+1}(0,s|q)=0.
$$
\subsection{A second proof of Proposition~\ref{prop:mathieu}}
 We now give a second proof of Proposition~\ref{prop:mathieu} using Theorem~\ref{thm:HH} and the orthogonality of the 
 continuous $q$-Hermite polynomials.
 Clearly Theorem~\ref{thm:HH} is equivalent to
\begin{align}\label{eq:keylink}
x^n=\sum_{k\equiv n\pmod{2}} c(n,k,q) s^{(n-k)/2}\tilde H_k(x,s|q).
\end{align}
To compute $c(n,k,q)$ we can take $s=1$ and let $\tilde  H_n(x|q)=\tilde H_n(x,s|q)$.
It is known (see \cite{ISV}) that  the continuous $q$-Hermite polynomials $(\tilde H_n(x|q))$ are orthogonal with respect to the linear functional $\varphi$ defined by
\begin{align}
\varphi(x^n)=\int_{-2/\sqrt{1-q}}^{2/\sqrt{1-q}}x^n v(x,q) dx,
\end{align}
where
$$
v(x,q) =\frac{\sqrt{(1-q)} (q)_\infty}{\sqrt{1-(1-q)x^2/4}4\pi}\prod_{k=0}^\infty \{1+(2-(1-q)x^2)q^k+q^{2k}\}.
$$
Since $\varphi((\tilde H_k(x|q))^2)= [k]_q!$,
it follows from \eqref{eq:keylink} that, for $k\equiv n\pmod{2}$,
\begin{align}\label{eq:computation}
c(n,k,q)=\frac{1}{ [k]_q!}\varphi(x^n\tilde H_k(x|q)).
\end{align}
Recall the well-known formula (see \cite{ISV}) 
\begin{align}\label{eq:2}
x^{2n}=\sum_{j=-n}^{n}{2n\choose n+j}T_{2j}(x/2),
\end{align}
where $T_n(\cos \theta)=\cos(n\theta)=T_{-n}(\cos \theta)$ is the $n$th Chybeshev polynomial of the  first kind.
By using the Jacobi triple product formula and the terminating $q$-binomial formula, 
we can prove  (see  \cite[p. 307]{Ismail}) that, for any integer $j$ and $a=\sqrt{1-q}$, 
\begin{align}\label{eq:ismail}
\varphi(T_{n-2j}(ax/2)\tilde H_{n}(x|q))=\frac{(-1)^{n+j}}{2a^n}
q^{n-j\choose 2} \{(q^{-n+j+1};q)_n+q^{n-j}(q^{-n+j}; q)_n\}.
\end{align}
It follows from \eqref{eq:computation}, \eqref{eq:2} and \eqref{eq:ismail} that 
\begin{align}
c(2n,2k,q)&=\frac{a^{-2n} }{[2k]_q!}\sum_{j= -n}^n {2n\choose n+j}\varphi(T_{2j}(ax/2)\tilde H_{2k}(x|q))\nonumber\\
&=\frac{(1-q)^{-(n-k)} }{(q;q)_{2k}}\sum_{j=-n}^n 
{2n\choose n+j}\frac{(-1)^{k+j}}{2}q^{k+j\choose 2}
\{(q^{-k-j+1};q)_{2k}+q^{k+j}(q^{-k-j};q)_{2k}\}\nonumber\\
\end{align}
Since  
$(q^{-k-j+1};q)_{2k}$  is zero if 
$j\not=-n,\ldots, -k$ and $j\not=k+1,\ldots, n$, and $(q^{-k-j};q)_{2k}$   is zero  if $j\not=-n,\ldots, -k-1$ or $j\not=k,\ldots, n$, 
we can split the last summation  into the following four summations:
\begin{align}
S_1&=\sum_{j=-n}^{-k}
{2n\choose n+j}\frac{(-1)^{k+j}}{2}q^{k+j\choose 2} (q^{-k-j+1};q)_{2k},\nonumber\\
S_2&=\sum_{j=k+1}^n 
{2n\choose n+j}\frac{(-1)^{k+j}}{2}q^{k+j\choose 2} (q^{-k-j+1};q)_{2k},\nonumber\\
S_3&=\sum_{j=-n}^{-k-1} 
{2n\choose n+j}\frac{(-1)^{k+j}}{2}q^{k+j\choose 2}
q^{k+j}(q^{-k-j};q)_{2k}, \nonumber\\
S_4&=\sum_{j=k}^n 
{2n\choose n+j}\frac{(-1)^{k+j}}{2}q^{k+j\choose 2}
q^{k+j}(q^{-k-j};q)_{2k}.\nonumber
\end{align}
It is readily seen,  by replacing $j$ by $-j$ in $S_1$  and $S_3$, that $S_1=S_4$ and $S_2=S_3$.
Therefore, 
\begin{align}
c(2n,2k,q)&=\frac{(1-q)^{-(n-k)} }{(q;q)_{2k}}(S_2+S_4)\nonumber\\
&=(1-q)^{-(n-k)} \sum_{j\geq 0} 
{2n\choose n+k+j}(-1)^{j}q^{j\choose 2}
\frac{[2k+2j]}{[2k+j]}{2k+j\brack j}.\label{eq:even}
\end{align}
This corresponds to \eqref{eq:lucasformula} for even indices. 
To derive the formula for odd indices we can use
 \eqref{eq:cig2} to get
\begin{align*}
c(2n+1,2k+1,q)=[2k+2]_qc(2n,2k+2,q)+c(2n,2k,q),
\end{align*}
and then apply \eqref{eq:even}. 

\subsection*{Acknowledgement}  The authors thank the two anonymous referees
 for useful comments on a previous version of this paper. The second author was  supported by 
the project PhysComb (ANR-08-Blan-0243-03).


\begin{thebibliography}{99}
\small \setlength{\itemsep}{-.8mm}
\bibitem{AAR}G. Andrews, R. Askey, and R. Roy, Special Functions,
Encyclopedia of Mathematics and Its Applications, 71, Cambridge
University Press, Cambridge, 1999.
\bibitem{aigner} M. Aigner, A course in enumeration, Graduate texts in Mathematics 238, Springer, 2007. 
\bibitem{Cigler03} J. Cigler, A new class of $q$-Fibonacci polynomials, Electr. J. Comb. 10(2003), \#R19.
\bibitem{Cigler09} J. Cigler,  $q$-Lucas polynomials and associated Rogers-Ramanujan type identities, arXiv:0907.0165.
 \bibitem{CJPR} S. Corteel, M. Josuat-Verg\`es, T. Prellberg, M. Rubey,
 Matrix Ansatz, lattice paths and rook placements, Proc. FPSAC'09.
 \bibitem{GR}G. Gasper and M. Rahman, {\it Basic Hypergeometric Series, Second Edition},
 Encyclopedia of Mathematics and Its Applications, Vol. 96,
 Cambridge University Press, Cambridge, 2004.
\bibitem{Ismail} M. E. H.  Ismail,
Classical and quantum orthogonal polynomials in one variable, Encyclopedia of Mathematics and its Applications, 98. Cambridge University Press, Cambridge, 2005.
\bibitem{IS} M. E. H.  Ismail, D. Stanton,
 More orthogonal polynomials as moments,
Mathematical essays in honor of Gian-Carlo Rota(Cambridge, MA, 1996),  377--396, Progr. Math., 161, Birkh\"auser Boston, Boston, MA, 1998.
\bibitem{ISV} 
M. E. H.  Ismail,  D.  Stanton, X. G. Viennot, 
 The combinatorics of $q$-Hermite polynomials and the Askey-Wilson integral,
 European J. Combin.  8  (1987),  no. 4, 379--392.
 \bibitem{JV} M. Josuat-Verg\`es,
 Rook placements in Young diagrams and permutation enumeration,
  arXiv:0811.0524v2.
\bibitem{KSZ}A. Kasraoui, D. Stanton, J. Zeng,
 The Combinatorics of Al-Salam-Chihara $q$-Laguerre polynomials, arXiv:0810.3232.
\bibitem{KZ} D. S. Kim and  J. Zeng, 
On a continued fraction formula of Wall,
Ramanujan J. 4 (2000), no. 4, 421--427.
\bibitem{KK} R. Koekoek and R. Swarttouw, The Askey scheme of hypergeometric orthogonal polynomials and its $q$-analogue, 
Report 98-17, TU Delft.
\bibitem{Pe} J. G. Penaud, Une preuve bijective d'une formule de Touchard-Riordan, Discrete Math., 139 (1995), 347-360.
\bibitem{Ri} J. Riordan, Combinatorial identities, John Wiley \& Sons, Inc., 1968.
\bibitem{Va} 
A. Varvak,  Rook numbers and the normal ordering problem,
J. Combin. Theory Ser. A 112 (2005), no. 2, 292--307. 
\bibitem{Vi} X. G. Viennot, \emph{Une th\'eorie combinatoire de polyn\^omes orthogonaux}, Lecture Notes, Universit\'e du Qu\'ebec \`a Montr\'eal 1984.
\bibitem{wall} H. S. Wall, Analytical theory of continued fractions, Chelsea, New York, 1967.
\bibitem{zeng} 
J. Zeng, The q-Stirling numbers, continued fractions and the q-Charlier and q-Laguerre 
polynomials, J. Comput. Appl. Math. 57 (1995), no. 3, 413-424.
\end{thebibliography}
\end{document}